\documentclass[12pt,a4paper]{article}
\pdfoutput=1 
\usepackage[a4paper,textwidth=160mm,textheight=230mm]{geometry}
\usepackage{amsmath,amssymb,amsthm}
\usepackage{mathtools}
\usepackage{bbm,latexsym}
\usepackage{graphicx} 
\usepackage{rotating} 
\usepackage[numbers,sort&compress]{natbib} 
\usepackage{authblk}
\usepackage{slashed} 
\usepackage{xcolor} 
\usepackage{makeidx} 
\usepackage{blkarray} 
\usepackage{listings}
\newtheorem{lemma}{Lemma}
\newtheorem{definition}{Definition}
\newtheorem{theorem}{Theorem}

\usepackage{xcolor}

\definecolor{codegreen}{rgb}{0,0.6,0}
\definecolor{codegray}{rgb}{0.5,0.5,0.5}
\definecolor{codepurple}{rgb}{0.58,0,0.82}
\definecolor{backcolour}{rgb}{0.95,0.95,0.92}

\lstdefinestyle{mystyle}{
    commentstyle=\color{codegreen},
    keywordstyle=\color{magenta},
    numberstyle=\tiny\color{codegray},
    stringstyle=\color{codepurple},
    basicstyle=\ttfamily\scriptsize,
    breakatwhitespace=false,         
    breaklines=true,                 
    captionpos=b,                    
    keepspaces=true,                 
    numbers=left,                    
    numbersep=5pt,                  
    showspaces=false,                
    showstringspaces=false,
    showtabs=false,                  
    tabsize=4
}

\allowdisplaybreaks

\begin{document}
\title{
\LARGE A normal form for bases of\\finite-dimensional vector spaces}
\setcounter{footnote}{2}
\author{Patrick Otto Ludl\thanks{E-mail: patrick.otto.ludl@itwm.fraunhofer.de}\\
{\small Fraunhofer Institute for Industrial Mathematics ITWM\\Fraunhofer-Platz 1, 67663 Kaiserslautern, Germany}}

\date{June 17, 2023}

\maketitle


\abstract{Most algorithms constructing bases of finite-dimensional vector spaces
return basis vectors which, apart from orthogonality, do not
show any special properties. While every basis is sufficient to define the vector space,
not all bases are equally suited to unravel properties of the
problem to be solved.
In this paper a normal form for bases of finite-dimensional
vector spaces is introduced which may prove very useful in the context
of understanding the structure of the problem in which the basis appears
in a step towards the solution. This normal form may be viewed as a new normal form
for matrices of full column rank.}

\section{Motivation}

Consider the following simple mathematical problem: Find a basis of the four-dimensional vector space $V \subset W \simeq \mathbbm{R}^5$
orthogonal to the vector $a = (1, 1, 1, 1, 1)^\text{T} \in W$. Numerically, an answer to this problem can be found
by computing a basis of the kernel (nullspace) of $a^\text{T}$. For example, Octave~\cite{octave}, using the command \verb+null+,
returns the following basis, denoted as columns of a matrix:
\begin{equation}\label{example_basis1}
 \left(
 \begin{smallmatrix*}[r]
  -0.44721 & -0.44721 & -0.44721 & -0.44721\\
   0.86180 & -0.13820 & -0.13820 & -0.13820\\
  -0.13820 &  0.86180 & -0.13820 & -0.13820\\
  -0.13820 & -0.13820 &  0.86180 & -0.13820\\
  -0.13820 & -0.13820 & -0.13820 &  0.86180
 \end{smallmatrix*}
 \right).
\end{equation}
To numerical accuracy this basis is a correct solution. Moreover, it is an orthogonal basis.
It is, however, not the answer a human would give. A human may give an answer like
\begin{equation}\label{example_basis2}
 \left(
 \begin{smallmatrix*}[r]
 1 & 0 & 0 & 0 \\
 -1 & 1 & 0 & 0 \\
 0 & -1 & 1 & 0 \\
 0 & 0 & -1 & 1\\
 0 & 0 & 0 & -1
 \end{smallmatrix*}
 \right).
\end{equation}
While the two bases of Eqs.~(\ref{example_basis1}) and~(\ref{example_basis2}) span the same vector space,
the second one appears to be simpler and more suited for the human mind to grasp the properties
of the solution. The basis in Eq.~(\ref{example_basis2}) has the following properties:
\begin{enumerate}
 \item Orthogonality to $a$ is evident without any computational aid (pen and paper, computer).
 \item Any basis vector can be characterised by two non-zero entries.\label{item_two_non_zero}
 \item Having seen the solution in this form, the solution can immediately be generalised to
the corresponding problem in $\mathbbm{R}^n$.
\end{enumerate}
Point~\ref{item_two_non_zero} above may be of crucial importance given the fact that in
the context of a scientific or technical application the
elements of the involved vector spaces
may have a \textit{specific physical meaning}. At this point it is important to remember
that the numerical notation of the vectors above is shorthand for
\begin{equation}
a = \left(
 \begin{smallmatrix*}[r]
 a_1 \\
 a_2 \\
 a_3 \\
 a_4 \\
 a_5
 \end{smallmatrix*}
 \right) \equiv \sum_{i=1}^5 a_i w_i,
\end{equation}
where the $w_i$ are \textit{fixed} basis vectors of $W$ chosen to describe the problem in an appropriate way.
Each of the $w_i$ may in general describe an individual physical quantity. For example, the $w_i$ could represent
measurements from five different devices.\footnote{The $w_i$ could even be functions, for example
linearly independent solutions of a homogeneous differential equation.}
If this is the case, each basis vector in Eq.~(\ref{example_basis2}) involves a \textit{minimal} number (two in the shown example)
of quantities, while the basis vectors of Eq.~(\ref{example_basis1}) combine all involved quantities in each basis vector.

The situation exemplified above may be faced in any situation which involves the
solution of a linear equation
\begin{equation}\label{__motiv_lin_eq}
M x = y,
\end{equation}
where $M\in\mathbbm{R}^{m\times n}$ is an $m\times n$-matrix of full row rank $m$ ($m<n$)
and $x\in \mathbbm{R}^n$, $y\in \mathbbm{R}^m$ are column vectors.
The general solution of this linear underdetermined system is given by
\begin{equation}\label{__motiv_sol}
x = x_p + \sum_{i=1}^{n-m} \lambda_i k_i,
\end{equation}
where $x_p$ is a particular solution of Eq.~(\ref{__motiv_lin_eq}),
$k_1,\ldots, k_{n-m} \in \mathbbm{R}^n$ form a basis of the kernel
$\mathrm{ker}\,M$ and $\lambda_1,\ldots,\lambda_{n-m} \in \mathbbm{R}$ are
free coefficients.
The choice of basis of $\mathrm{ker}\,M$ is not restricted, and
regarding the parameterisation of the solution~(\ref{__motiv_sol}) each
basis is equivalent. Since, however, $\mathrm{ker}\,M \subset \mathbbm{R}^n$
is \textit{embedded} into an $n$-dimensional vector space where
each of the $n$ directions may have a different \textit{physical meaning},
some bases of $\mathrm{ker}\,M$ may be more expressive or suited to the
problem than others.

The outline of this paper is as follows. In Section~\ref{sec_normal_form} a normal form for bases of
finite-dimensional vector spaces is defined. The practical use of this normal form is demonstrated by sample applications in
Section~\ref{sec_applications}. Finally, Section~\ref{sec_conclusions} concludes the paper.

\section{Definition of a normal form for bases of finite-dimensional vector spaces}\label{sec_normal_form}

In the following, a definition of a normal form for a basis of a finite-dimensional
vector space $V \subseteq \mathbbm{F}^m$ with $\mathbbm{F}=\mathbbm{R}$ or $\mathbbm{C}$ is given.
It has been developed with the following goals in mind:
\begin{itemize}
 \item Existence for any vector space $V \subseteq \mathbbm{F}^m$ of dimension $n=\mathrm{dim}\,V \geq 2$.
 \item Uniqueness.
 \item A high number of zero entries in the basis vectors.
 \item If $\mathrm{dim} V = m$, the basis in normal form shall consist of the columns of the $m\times m$ unit matrix.
\end{itemize}

\paragraph{Given:} An $m \times n$-matrix
\begin{equation}\label{eq_columns_of_A}
 A = \begin{pmatrix*}
 c_1,\ldots,c_n
\end{pmatrix*} =
\begin{pmatrix*}
r_1\\
\vdots\\
r_m
\end{pmatrix*} \in \mathbbm{F}^{m\times n}
\end{equation}
of full column rank $n\leq m$. The columns $c_1,\ldots,c_n$ represent the basis vectors
of the $n$-dimensional vector space $V\subseteq \mathbbm{F}^m$ for which a basis in normal form shall be constructed.
To avoid trivial cases, $n\geq 2$ is assumed.

Let $R=\{r_1,\ldots,r_m\}$ be the set of rows of $A$.
Consider now all selections $S \subset R$ of rows s.t.\ the selected rows span a vector space of
dimension $n-1$, i.e.\
\begin{equation}\label{eq_dim_span_S}
 \mathrm{dim}\;\mathrm{span} (S) = n-1.
\end{equation}
Each of the vector spaces $\mathrm{span}(S) \subset \mathbbm{F}^n$, geometrically is
an $(n-1)$-di\-men\-sio\-nal hyperplane
through $0$ which contains the selected row vectors. Let $s\neq 0$ be a normal vector of
this hyperplane,\footnote{The definition of orthogonality
in vector spaces isomorphic to $\mathbbm{C}^n$ usually is based
on the scalar product $\langle r\vert s\rangle \equiv r^\dagger s$. Nevertheless, here
$r \cdot s = 0$ is used as the defining property of a ``normal'' vector $s$ also
in case of a complex vector space.} \textit{i.e.}\ $r \cdot s = 0\; \forall r\in S$.
Then $As \in \mathbbm{F}^m$ contains at least $|S|$ zero entries.
Since $s\neq 0$ and $A$ has full column rank, $As$ has at least one non-vanishing entry.
Therefore, the normal vector $s$ can be made unique by requiring that the first non-vanishing element of $As$
shall be $+1$. In the following, the normal vectors $\hat{s}$ normalised s.t.\ they fulfil this requirement
are denoted with a hat. In this way, each selection $S$ is assigned a unique vector $\hat{s}\in \mathbbm{F}^n$.
\begin{lemma}\label{lemma_1}
The set of all normal vectors $\hat{s}\in \mathbbm{F}^n$ to all selections $S$ is a generating set of $\mathbbm{F}^n$.
\end{lemma}
\begin{proof}
This can be proven by explicitly constructing $n$ linearly independent vectors $\hat{s}$.
Since $A$ has full column rank $n$, one can select $n$ linearly independent rows $\rho_1,\ldots, \rho_n$ of
$A$. From these, one can form the $n$ different selections
\begin{equation}
S_i = \{\rho_1,\ldots, \rho_{i-1}, \rho_{i+1},\ldots ,\rho_n\}.
\end{equation}
Each of the $S_i$ spans an $(n-1)$-dimensional vector space. The $i$-th column of
\begin{equation}
\rho^{-1} \equiv \begin{pmatrix*}
\rho_1\\
\vdots\\
\rho_n
\end{pmatrix*}^{-1}
\end{equation}
is a normal vector of $\mathrm{span}(S_i)$. Therefore, the $n$ linearly independent columns
of $\rho^{-1}$ provide $n$ linearly independent vectors $s$. Normalisation of these
yields $n$ linearly independent vectors $\hat{s}$.
\end{proof}
The vectors $\hat{s}$ will be used to define the normal form.
The main idea of this definition is to \textit{order} all possible
$\hat{s}$ w.r.t.\ the number of zero entries they generate in $A\hat{s}$.
For the same number of zeros, those $A\hat{s}$ shall be preferred for
which the zero entries have higher indices.
For example,
\begin{equation}
\begin{pmatrix*}
0\\
0\\
1\\
\times
\end{pmatrix*}
\quad
\text{shall be preferred over}
\quad
\begin{pmatrix*}
1\\
\times\\
\times\\
0
\end{pmatrix*}
\end{equation}
and
\begin{equation}
\begin{pmatrix*}
1\\
\times\\
0\\
\times
\end{pmatrix*}
\quad
\text{over}
\quad
\begin{pmatrix*}
0\\
1\\
\times\\
\times
\end{pmatrix*}.
\end{equation}
The symbol $\times$ here denotes an arbitrary non-zero entry.
This ordering will be formally defined by assigning each vector $\hat{s}$
an integer number.
\begin{definition}\label{def_1}
Let $\hat{s}$ be the normal vector to the hyperplane $\mathrm{span}(S)$,
then define
\begin{equation}
\theta^{(1)}(\hat{s}) \equiv 2^{m+\sum_{j=1}^m \theta_j},
\end{equation}
\begin{equation}
\theta^{(2)}(\hat{s}) \equiv \sum_{j=1}^m \theta_j 2^{j-1},
\end{equation}
\text{where}
\begin{equation}
\theta_j = \begin{cases}
1 & \text{for $(A\hat{s})_j =0$} \\
0 & \text{else}.
\end{cases}
\end{equation}
\end{definition}
The first term $\theta^{(1)}(\hat{s})$ is larger than the maximal value $\theta^{(2)}$
can assume (which is $2^m-2$ if all entries of $A\hat{s}$ except the first one are zero)
and encodes the total number of zeros in $A\hat{s}$. The second term describes the
preference for zero entries with higher indices.
Therefore, the number
\begin{equation}
\theta(\hat{s}) \equiv \theta^{(1)}(\hat{s}) + \theta^{(2)}(\hat{s})
\end{equation}
is consistent with the ordering for $\hat{s}$ sketched above.
Note that $\theta(\hat{s})$ equals the binary number $(10\ldots0\theta_m \theta_{m-1} \theta_{m-2} \cdots \theta_2 \theta_1)_2$,
where the first digit 1 encodes $\theta^{(1)}(\hat{s})$.
The following properties of $\theta(\hat{s})$ are important:
\begin{itemize}
 \item The encoding of the information
which elements of $A\hat{s}$ vanish as the number $\theta(\hat{s})$ is bijective.
 \item Different $\hat{s}$ correspond to different hyperplanes and thus to
different sets of vanishing elements of $A\hat{s}$. Since the encoding
is bijective, this means that \textit{different $\hat{s}$ are assigned different $\theta(\hat{s})$.}
 \item By definition of $S$, the set of rows orthogonal to $\hat{s}$ contains $n-1$ linearly independent vectors,
and thus unambiguously defines $\hat{s}$.
But two different selections $S_1$, $S_2$ of rows of $A$ 
can correspond to the same $\hat{s}$ and thus to the same $\theta(\hat{s})$.
This is the case if and only if
\begin{equation}
\mathrm{span}(S_1) = \mathrm{span}(S_2).
\end{equation}
A simple example
illustrating this is the matrix
\begin{equation}
A=\left(\begin{smallmatrix*}
r_1\\r_2\\r_3
\end{smallmatrix*}\right)=\left(\begin{smallmatrix*}
1 & 0\\
0 & 1\\
2 & 0
\end{smallmatrix*}\right)
\end{equation}
with the three possible selections $S_1=\{r_1\}$, $S_2=\{r_2\}$, $S_3=\{r_3\}$.
Though $S_1$ and $S_3$ are different, they correspond to the same $\hat{s}\equiv(0,\,1)^\mathrm{T}$. The
$\hat{s}$ of $S_2$ is $\hat{s}^\prime\equiv(1,\,0)^\mathrm{T}$. From
\begin{equation}
A\hat{s}=\left(\begin{smallmatrix*}
0\\1\\0
\end{smallmatrix*}\right) \text{ and } A\hat{s}^\prime=\left(\begin{smallmatrix*}
1\\0\\2
\end{smallmatrix*}\right)
\end{equation}
one finds
\begin{equation}
\theta(\hat{s})=2^{3+2}+2^0+2^2=37 \text{ and } \theta(\hat{s}^\prime)=2^{3+1}+2^1=18.
\end{equation}
\end{itemize}
Since different $\hat{s}$ are assigned different $\theta(\hat{s})$, the numbers can be used to order the
set of all normal vectors $\hat{s}$ constructed from all selections $S$.
\begin{definition}\label{definition_normal_form}
Let $\mathcal{S}$ be the set of all different normal vectors $\hat{s}$ to all selections $S \subset R$.
Then define
\begin{subequations}\label{eq_def_normalform}
\begin{align}
& \hat{s}_1 \equiv \underset{\hat{s}\in\mathcal{S}}{\mathrm{argmax}} \,\theta(\hat{s}),\\
& \hat{s}_j \equiv \underset{\hat{s}\in\mathcal{S}_j}{\mathrm{argmax}} \,\theta(\hat{s}),\; j\in \{2,\ldots,n\},
\end{align}
\end{subequations}
where
\begin{equation}
\mathcal{S}_j \equiv \{ \hat{s} \in \mathcal{S} \vert \text{$\hat{s}$ linearly independent of $\hat{s}_1,\ldots, \hat{s}_{j-1}$} \}.
\end{equation}
The normal form of the basis specified by the columns $\{c_1,\ldots,c_n\}$ of $A$ is defined as
the set of column vectors
\begin{equation}
\{ A\hat{s}_1,\ldots A\hat{s}_n\}.
\end{equation}
\end{definition}
\begin{theorem}\label{theorem_1}
The normal form exists and is unique for all matrices $A$ of Eq.~(\ref{eq_columns_of_A}), i.e.\ all vector spaces
$V\subset \mathbbm{F}^m$ of dimension $n \geq 2$. If $n = m$, i.e.\ if $A$ is a square matrix, the normal form of $A$
is given by the $m\times m$ unit matrix.
\end{theorem}
\begin{proof}
From Lemma~\ref{lemma_1} it follows that there are $n$ linearly independent
vectors $\hat{s}$. Since all $\hat{s}$ are mapped to different $\theta(\hat{s})$,
the maximisations of Definition~\ref{definition_normal_form} all have unique solutions.
Therefore, the normal form exists and is unique.
In the special case of a square matrix $A$,
there are only $m$ different selections $S$
of rows of $A$ and the corresponding vectors $\hat{s}$
must be the columns of $A^{-1}$---see the proof of Lemma~\ref{lemma_1}---and
one has
\begin{equation}
A(\hat{s}_1,\ldots,\hat{s}_m)=AA^{-1}=\mathbbm{1}_{m}.
\end{equation}
The vectors $\hat{s}_1,\ldots \hat{s}_m$ according to this equation are ordered
as in Eq.~(\ref{eq_def_normalform}) because
\begin{equation}
\theta(\hat{s}_j) = 2^{2m-1} + 2^{m}-1-2^{j-1} > \theta(\hat{s}_{j+1}).
\end{equation}
\end{proof}

\section{Example applications}\label{sec_applications}

The following subsections illustrate the use of the normal
form in examples from different areas of
physics and mathematics. All numerical computations of
the normal form have been carried out using the implementation
provided in appendix~\ref{sec_python_implementation}.

\subsection{Basis of a space orthogonal to a vector}

As a first example consider the basis in Eq.~(\ref{example_basis1}) of the vector space $V$ orthogonal
to $a = (1, 1, 1, 1, 1)^\text{T}$. In normal form the basis is given by
\begin{equation}
 \left(
 \begin{smallmatrix*}[r]
 1 & 1 & 1 & 1 \\
 -1 & 0 & 0 & 0 \\
 0 & -1 & 0 & 0 \\
 0 & 0 & -1 & 0\\
 0 & 0 & 0 & -1
 \end{smallmatrix*}
 \right),
\end{equation}
which shows the same desirable properties as the basis in Eq.~(\ref{example_basis2})
which was the motivation for the need of a normal form.
This may be generalised to an arbitrary $a \in \mathbbm{F}^m$, provided $a_i\neq 0 \,\forall i\in \{1,\ldots,m\}$.
Then the normal form is given by
\begin{equation}\label{nf_orthogonal_vector1}
 \left(
 \begin{smallmatrix*}[r]
 1 & 1 & \cdots & 1 \\
 -a_1/a_2 & 0 &  & 0 \\
 0 & -a_1/a_3 &  & 0 \\
 0 & 0 &  & 0\\
 \vdots & \vdots & & \vdots\\
 0 & 0 & \cdots & -a_1/a_m
 \end{smallmatrix*}
 \right)
\end{equation}
or, in terms of the basis vectors $v_1,\ldots,v_{m-1}$
of $\mathrm{ker} (a^\mathrm{T})$,
\begin{equation}\label{nf_orthogonal_vector2}
(v_i)_j = \begin{cases}
1 & \text{for } j=1\\
-a_1/a_j & \text{for } j=i+1\\
0 & \text{else}.
\end{cases}
\end{equation}

Also for the case of a vector space $V$ orthogonal to two vectors
$a_1$, $a_2 \in \mathbbm{F}^m$ an analytic expression for the basis in normal form
can be given. Note that the expression constructed in the following is \textit{not}
valid in all cases.
Nevertheless its derivation serves as a good illustration for the general structure
of the normal form and special cases to be taken care of.
As for the one-dimensional case, the first row of the normal form
in the general case will
consist of one entries only. Then, according to the requirement of
putting the zero entries as low as possible, also the second row
in general will be non-zero and the remainder of the $m\times m-2$-matrix
of basis vectors will be diagonal, \textit{i.e.}\
\begin{equation}\label{nf_orthogonal_2vectors1}
 \left(
 \begin{smallmatrix*}[c]
 1 & 1 & \cdots & 1 \\
 \times & \times & \cdots & \times \\
 \times & 0 &  & 0 \\
 0 & \times &  & 0 \\
 \vdots & \vdots & & \vdots\\
 0 & 0 & \cdots & \times
 \end{smallmatrix*}
 \right),
\end{equation}
where $\times$ denotes a non-zero entry. From the requirements $v_i \cdot a_1=0$
and $v_i \cdot a_2=0$ for the basis vectors $v_i$ in normal form then follows
\begin{equation}\label{nf_orthogonal_2vectors2}
(v_i)_j = \begin{cases}
1 & \text{for } j=1\\
v^{(i)} & \text{for } j=2\\
v^{(i)\prime} & \text{for } j=i+2\\
0 & \text{else}
\end{cases}
\end{equation}
with
\begin{equation}\label{nf_orthogonal_2vectors2b}
\begin{pmatrix*}
v^{(i)}\\
v^{(i)\prime}
\end{pmatrix*} =
-\begin{pmatrix*}
(a_{1})_2 & (a_{1})_{i+2}\\
(a_{2})_2 & (a_{2})_{i+2}
\end{pmatrix*}^{-1}
\begin{pmatrix*}
(a_{1})_1\\
(a_{2})_1
\end{pmatrix*}.
\end{equation}
This result has been tested numerically with real and complex
random vectors $a_1$ and $a_2$ up to dimension 20.
However, Eqs.~(\ref{nf_orthogonal_2vectors1}) and~(\ref{nf_orthogonal_2vectors2})
can only hold if the $2\times 2$-matrix in Eq.~(\ref{nf_orthogonal_2vectors2b}) is invertible
for all $i\in \{1,\ldots,m-2\}$. Moreover, in order to describe a basis, the matrix of
Eq.~(\ref{nf_orthogonal_2vectors1}) must not contain more than two vanishing (or three linearly dependent)
rows. Thus, if all $v^{(i)}$ were equal, at most one of the $v^{(i)\prime}$ could be zero.
And if the $v^{(i)}$ are not all equal, at most two of the $v^{(i)\prime}$ can be zero.
Otherwise, the normal form will not be given by Eq.~(\ref{nf_orthogonal_2vectors2})
and one has to resort to Definition~\ref{definition_normal_form} to compute the normal
form. Appendix~\ref{sec:algorithm} provides an algorithm to compute the normal form
which is valid for all cases.

\subsection{Dimensional analysis}

The main statement of dimensional analysis~\cite{Buckingham, Bridgman} can be summarised as:
If a physical quantity $q$ can be expressed as a function of other physical quantities 
$q_1,\ldots q_n$, then this function assumes the form
\begin{equation}
q = \tilde{q} \times f(a_1,\ldots,a_m),
\end{equation}
where $\tilde{q}$ is a product of powers of $q_1,\ldots,q_n$ of the physical dimension
of $q$ and $\{a_1,\ldots,a_m\}$ is a set of independent dimensionless products of
powers of $q_1,\ldots,q_n$.

As an example consider the equation of motion of a one-dimensional
harmonic oscillator
\begin{equation}\label{EOM_harm_oszi}
m \ddot{x} = -kx,\quad x(0) = x_0,\quad \dot{x}(0) = \dot{x}_0.
\end{equation}
The position $x(t)$ of the oscillator must then be a function of $t, x_0, \dot{x}_0, k$ and $m$.
Applying dimensional analysis to this problem one finds
\begin{equation}
x(t) = x_0 \times f(a_1,\ldots,a_m),
\end{equation}
where $x_0$ has been chosen as a quantity of the same dimension as $x(t)$.
A complete set of independent dimensionless products of powers $\{a_1,\ldots,a_m\}$
can be found by means of the matrix~\cite{Buckingham, Bridgman}
\begin{equation}
B = \begin{blockarray}{cccccc}
& t & x_0 & \dot{x}_0 & k & m \\
\begin{block}{l(rrrrr)}
M & 0 & 0 & 0 & 1 & 1\\
L & 0 & 1 & 1 & 0 & 0\\
T & 1 & 0 & -1 & -2 & 0\\
\end{block}
\end{blockarray}\;.
\end{equation}
The rows of $B$ encode the dimensions $M, L, T = $ mass, length, time and the columns
the five physical quantities the sought-for $x(t)$ depends on.
The matrix entries are the powers of $M, L, T$ in the physical dimensions of $t, x_0, \dot{x}_0, k, m$.
A maximal set of independent dimensionless products of powers $a_1,\ldots a_m$
can then be found by constructing the basis of $\mathrm{ker}\,B$~\cite{Buckingham, Bridgman}.
Doing so via singular value decomposition of $B$ using the NumPy package~\cite{numpy}
one finds that $\mathrm{ker}\,B$ is spanned by the columns of
\begin{equation}\label{ker_B_numpy}
 \left(
 \begin{smallmatrix*}[r]
 -0.35314643 & -0.76783678 \\
  0.64522571 & -0.11070323 \\
 -0.64522571 &  0.11070323 \\
  0.14603964 & -0.43927000 \\
 -0.14603964 &  0.43927000
 \end{smallmatrix*}
 \right).
\end{equation}
The normal form of this basis is given by
\begin{equation}\label{ker_B_nf}
 \left(
 \begin{smallmatrix*}[r]
 1 & 1\\
-1 & 0\\
 1 & 0\\
 0 & 1/2\\
 0 & -1/2
 \end{smallmatrix*}
 \right).
\end{equation}
A set of independent dimensionless products of powers of 
$t, x_0, \dot{x}_0, k, m$ is thus given by
\begin{equation}
a_1 = \frac{t \dot{x}_0}{x_0},\quad a_2 = \sqrt{\frac{k}{m}} t.
\end{equation}
According to dimensional analysis, the solution of the equation of motion
of the oscillator must therefore assume the form
\begin{equation}\label{form_solution_harmonic_oscillator}
x(t) = x_0 \times f(a_1, a_2).
\end{equation}
Indeed, the solution
\begin{equation}
\begin{split}
x(t) & = x_0 \cos(\omega t) + \frac{\dot{x}_0}{\omega} \sin(\omega t)
 = x_0 \times \left( \cos(\omega t) + \frac{t\dot{x}_0}{x_0} \frac{1}{\omega t} \sin(\omega t)\right) =\\
& = x_0 \times \left( \cos(a_2) +  \frac{a_1}{a_2}\sin(a_2) \right),
\end{split}
\end{equation}
with $\omega \equiv \sqrt{k/m}$ is of the form of Eq.~(\ref{form_solution_harmonic_oscillator}).

In this example, the normal form allows to formulate the solution of a physical
problem in a particularly simple form. The reason for this simplicity
is that in normal form the dimensionless products
of powers $a_i$ depend on as few physical quantities as possible.
The idea to use the freedom of basis choice for this purpose
has for example been applied in the symbolic regression technique
AI Feynman~\cite{AIFeynman}.

\subsection{Numerical application of Noether's theorem}\label{sec:Noether}

In its simplest form Noether's theorem~\cite{Noether} states that for every
differentiable one-parametric symmetry
transformation of the Lagrangian (to be more precise: the action integral) of a system,
there is a corresponding conserved quantity.
In the following, the case of a Lagrangian $L(x(t), \dot{x}(t))$
invariant under a one-parametric transformation $x(t) \mapsto x_\lambda(t)$ with
$x_\lambda(t)\vert_{\lambda = 0} = x(t)$ is considered. Due to the invariance
\begin{equation}
L(x_\lambda(t), \dot{x}_\lambda(t)) = L(x(t),\dot{x}(t))
\end{equation}
one finds\footnote{To simplify the notation, in this subsection sums and indices
are suppressed in those places where the notation stays unambiguous.
For example, $\frac{\partial L}{\partial x} \frac{\partial x_\lambda}{\partial \lambda}$
is shorthand for $\sum_{i=1}^p \frac{\partial L}{\partial x_i} \frac{\partial (x_\lambda)_i}{\partial \lambda}$,
where $p$ is the number of components of the vector $x$.}
\begin{equation}
\begin{split}
0 & = \frac{d}{d\lambda}\Big\vert_{\lambda=0} L(x_\lambda(t), \dot{x}_\lambda(t)) = \frac{\partial L}{\partial x} \frac{\partial x_\lambda}{\partial \lambda}\Big\vert_{\lambda=0} + 
\frac{\partial L}{\partial \dot{x}} \frac{\partial \dot{x}_\lambda}{\partial \lambda}\Big\vert_{\lambda=0}\\
& = \frac{\partial L}{\partial x} \frac{\partial x_\lambda}{\partial \lambda}\Big\vert_{\lambda=0} + 
\frac{\partial L}{\partial \dot{x}} \frac{d}{dt} \frac{\partial x_\lambda}{\partial \lambda}\Big\vert_{\lambda=0}\\
& = \underbrace{\left(\frac{\partial L}{\partial x} - \frac{d}{dt} \frac{\partial L}{\partial \dot{x}}\right)}_{0} \frac{\partial x_\lambda}{\partial \lambda}\Big\vert_{\lambda=0} + 
\frac{d}{dt}\left(\frac{\partial L}{\partial \dot{x}} \frac{\partial x_\lambda}{\partial \lambda}\right)\Big\vert_{\lambda=0}\\
& = \frac{d}{dt}\left(\frac{\partial L}{\partial \dot{x}} \frac{\partial x_\lambda}{\partial \lambda}\right)\Big\vert_{\lambda=0},
\end{split}
\end{equation}
where in the last step the Euler-Lagrange equation has been used. For an infinitesimal symmetry transformation
of the form
\begin{equation}\label{linear_symmetry_transformation}
x_\lambda(t) = x(t) + \lambda (a + Bx(t)) + O(\lambda)^2 
\end{equation}
the conserved quantity is thus given by
\begin{equation}\label{conserved_quantity}
\frac{\partial L}{\partial\dot{x}} (a+B x(t)).
\end{equation}
In the following, the normal form for bases will be shown to provide an
elegant way to numerically apply Noether's theorem to symmetry transformations
of this form.

To find all symmetry transformations of the form of Eq.~(\ref{linear_symmetry_transformation}),
expand $L(x_\lambda(t), \dot{x}_\lambda(t))$ into a Taylor series around $\lambda = 0$.
\begin{equation}
L(x_\lambda, \dot{x}_\lambda) = L(x, \dot{x}) +
\lambda \left(\frac{\partial L}{\partial x}(a+Bx) + \frac{\partial L}{\partial \dot{x}} B\dot{x} \right) + O(\lambda^2).
\end{equation}
The transformation is a symmetry transformation if
\begin{equation}\label{symmetry_condition}
\frac{\partial L}{\partial x}(a+Bx) + \frac{\partial L}{\partial \dot{x}} B\dot{x} = 0 \quad \forall x, \dot{x}.
\end{equation}
This provides a set of (infinitely many) linear equations for $a$ and $B$.
Let the coordinate vector $x$ have $p$ components, i.e.\ $x\in \mathbbm{R}^p$.
Then, defining
\begin{equation}\label{def_DL}
\mathcal{D} L \equiv
\begin{pmatrix*}
\frac{\partial L}{\partial x_1}, \ldots,
\frac{\partial L}{\partial x_p},
\frac{\partial L}{\partial x_1} x_1 + 
\frac{\partial L}{\partial \dot{x}_1} \dot{x}_1, 
\frac{\partial L}{\partial x_1} x_2 + 
\frac{\partial L}{\partial \dot{x}_1} \dot{x}_2, \ldots
\frac{\partial L}{\partial x_p} x_p + 
\frac{\partial L}{\partial \dot{x}_p} \dot{x}_p
\end{pmatrix*},
\end{equation}
and
\begin{equation}
A \equiv
\begin{pmatrix*}
a_1,\ldots, a_p, B_{11}, B_{12}, \ldots B_{pp}
\end{pmatrix*}^\mathrm{T},
\end{equation}
the equations~(\ref{symmetry_condition}) can be written as
\begin{equation}
\mathcal{D}L(x,\dot{x}) A = 0 \quad \forall x, \dot{x}.
\end{equation}
The possible values for $A$ can thus be found by constructing a basis of the
common kernel
\begin{equation}\label{eq_common_kernel}
\bigcap_{x, \dot{x}}\; \mathrm{ker}\,\mathcal{D}L(x,\dot{x}).
\end{equation}
This common kernel can be found numerically
by choosing $\nu$ random values for $x$ and $\dot{x}$.
Combining the obtained random $\mathcal{D}L(x,\dot{x})$
to a $\nu \times (p+p^2)$-matrix,
a basis of the common kernel is given by a basis of the kernel
of the matrix.
If there are symmetry transformations of the proposed form,
this procedure for a sufficiently high $\nu$ will lead to a non-trivial
common kernel valid for all $x$ and $\dot{x}$.

As an example consider the two-body problem in three dimensions
with a $1/r$ potential:
\begin{equation}
L(\vec{x}_1, \vec{x}_2, \dot{\vec{x}}_1, \dot{\vec{x}}_2) =
\frac{1}{2}m_1 \dot{\vec{x}}_1^2 +
\frac{1}{2}m_2 \dot{\vec{x}}_2^2 +
\frac{\alpha}{|\vec{x}_1-\vec{x}_2|}.
\end{equation}
With $\vec{r}\equiv \vec{x}_1-\vec{x}_2$ and $r\equiv |\vec{r}|$, the derivatives needed for Eq.~(\ref{def_DL})
are
\begin{equation}
\frac{\partial L}{\partial (\vec{x}_1)_i} = -\frac{\alpha r_i}{r^3}, \quad
\frac{\partial L}{\partial (\vec{x}_2)_i} = +\frac{\alpha r_i}{r^3}, \quad
\frac{\partial L}{\partial (\dot{\vec{x}}_1)_i} = m_1 (\dot{\vec{x}}_1)_i, \quad
\frac{\partial L}{\partial (\dot{\vec{x}}_2)_i} = m_2 (\dot{\vec{x}}_2)_i
\end{equation}
for $i=1,2,3$.
The vector $\mathcal{D}L$ in this example has dimension $6+6^2=42$. Assigning $m_1, m_2, \alpha$ fixed numerical values
$\in [0, 1]$ (the exact values are not relevant, however, $m_1\neq m_2$ is important) and
subsequently assigning $\vec{x}_1$, $\vec{x}_2$, $\dot{\vec{x}}_1$, $\dot{\vec{x}}_2$ different
random values\footnote{In the numerical experiment $\nu=100$ random values
were used. This was sufficient to obtain the kernel valid for all
$x, \dot{x}$.} in $[0, 1]^3$, and computing the common kernel of all resulting values for $\mathcal{D}L$,
a nine-dimensional common kernel is found. Numerically, the basis of this kernel is given
by a $42\times 9$-matrix. In the performed numerical experiment, 12 of the 42 rows of this matrix
vanished, the others, however, were densely populated with non-zero entries.
The normal form of the basis is given by a $42\times 9$-matrix with only
36 non-zero entries. The found basis and its normal form are visualised in
Fig.~\ref{fig_basis_nf_noether}.
\begin{figure}
\begin{center}
\includegraphics[height=0.3\textheight]{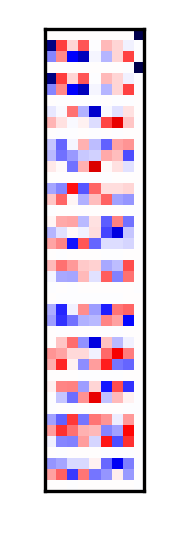}
\quad
\includegraphics[height=0.3\textheight]{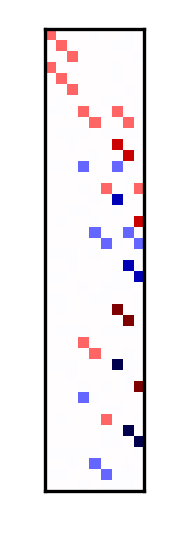}
\end{center}
\caption{The numerically found basis of the common kernel~Eq.~(\ref{eq_common_kernel}) (left)
and its normal form (right). The coloured fields indicate the values of the
matrix entries, white representing a vanishing entry.}\label{fig_basis_nf_noether}
\end{figure}

In terms of $a$ and $B$, the nine basis vectors in normal form are given by
\begin{subequations}\label{kernel_basis_noether}
\begin{align}
& a_1 = \begin{pmatrix*}
1 & 0 & 0 & 1 & 0 & 0
\end{pmatrix*}^\mathrm{T}, \quad
B_1 = 0, \\
& a_2 = \begin{pmatrix*}
0 & 1 & 0 & 0 & 1 & 0
\end{pmatrix*}^\mathrm{T}, \quad
B_2 = 0, \\
& a_3 = \begin{pmatrix*}
0 & 0 & 1 & 0 & 0 & 1
\end{pmatrix*}^\mathrm{T}, \quad
B_3 = 0, \\
& a_4 = 0, \quad
B_4 = \begin{pmatrix*}
T_3 & 0\\
0 & T_3
\end{pmatrix*}, \\
& a_5 = 0, \quad
B_5 = \begin{pmatrix*}
T_2 & 0\\
0 & T_2
\end{pmatrix*}, \\
& a_6 = 0, \quad
B_6 = \begin{pmatrix*}
T_1 & 0\\
0 & T_1
\end{pmatrix*}, \\
& a_7 = 0, \quad
B_7 = \begin{pmatrix*}
T_3 & \beta T_3\\
(1+\beta)T_3 & 0
\end{pmatrix*}, \label{eq:a7}\\
& a_8 = 0, \quad
B_8 = \begin{pmatrix*}
T_2 & \beta T_2\\
(1+\beta)T_2 & 0
\end{pmatrix*}, \label{eq:a8}\\
& a_9 = 0, \quad
B_9 = \begin{pmatrix*}
T_1 & \beta T_1\\
(1+\beta)T_1 & 0
\end{pmatrix*}, \label{eq:a9}
\end{align}
\end{subequations}
where
\begin{equation}\label{eq:def_beta}
\beta \equiv \frac{m_2}{m_1-m_2} = \frac{1}{\frac{m_1}{m_2}-1}
\end{equation}
and
\begin{equation}
T_1 = \begin{pmatrix*}
0 & 0 & 0\\
0 & 0 & 1\\
0 & -1 & 0
\end{pmatrix*}, \quad
T_2 = \begin{pmatrix*}
0 & 0 & 1\\
0 & 0 & 0\\
-1 & 0 & 0
\end{pmatrix*}, \quad
T_3 = \begin{pmatrix*}
0 & 1 & 0\\
-1 & 0 & 0\\
0 & 0 & 0
\end{pmatrix*}
\end{equation}
are the generators of infinitesimal rotations about the three axes.
Note that in the numerical experiment described above, $\beta$ is of course represented by a specific floating point number.
From dimensional analysis ($B$ is dimensionless) and repeating the experiment for different values
assigned to $m_1$ and $m_2$ the functional form Eq.~(\ref{eq:def_beta}) of $\beta$
can be inferred.\footnote{Remark: Eqs.~(\ref{kernel_basis_noether}) are valid only for $m_1\neq m_2$. For $m_1=m_2$ the
basis vectors 1 to 3 remain unchanged, and the remaining six vectors are given by $a=0$ and
$\left(\begin{smallmatrix*}
0 & T_3\\
T_3 & 0
\end{smallmatrix*}\right)$, $B_4$,
$\left(\begin{smallmatrix*}
0 & T_2\\
T_2 & 0
\end{smallmatrix*}\right)$,
$\left(\begin{smallmatrix*}
0 & T_1\\
T_1 & 0
\end{smallmatrix*}\right)$, $B_5$, $B_6$.}

The physical interpretation of six of these nine symmetry transformations is straight forward.
The first three correspond to translations in the directions of the three axes in space
and the corresponding conserved quantities, cf.~Eq.~(\ref{conserved_quantity}), are the components of the total momentum
$\vec{p} = m_1 \dot{\vec{x}}_1 + m_2 \dot{\vec{x}}_2$. The symmetry transformations number
4 to 6 correspond to rotations about the three axes through the origin, resulting
in conservation of the total angular
momentum $\vec{L} = m_1 \vec{x}_1\times \dot{\vec{x}}_1 + m_2 \vec{x}_2\times \dot{\vec{x}}_2$.

The remaining three symmetry transformations 7 to 9 are unusual in the sense that they \textit{mix}
the coordinates of the two bodies, \textit{i.e.}\ the matrices $B_7$ to $B_9$
are not block-diagonal. They can, however, be simultaneously block-diagonalised by
changing to center-of-mass and relative coordinates, \textit{i.e.}
\begin{equation}
\begin{pmatrix*}
\vec{x}_\text{CM}\\
\vec{r}
\end{pmatrix*} \equiv 
\begin{pmatrix*}
\frac{m_1}{m_1+m_2}\vec{x}_1 +\frac{m_2}{m_1+m_2}\vec{x}_2\\
\vec{x}_1-\vec{x}_2
\end{pmatrix*} =
\begin{pmatrix*}
\frac{1+\beta}{1+2\beta}\mathbbm{1} & \frac{\beta}{1+2\beta}\mathbbm{1}\\
\mathbbm{1} & -\mathbbm{1}
\end{pmatrix*}
\begin{pmatrix*}
\vec{x}_1\\
\vec{x}_2
\end{pmatrix*}.
\end{equation}
In these new coordinates, the matrices $B_7$ to $B_9$ become
\begin{equation}\label{eq:transformation_cm_r}
\begin{pmatrix*}
(1+\beta)T & 0\\
0 & -\beta T
\end{pmatrix*}\quad\text{with}\quad T=T_3, T_2, T_1.
\end{equation}
Thus, the corresponding symmetry transformations are infinitesimal rotations
of the center of mass $\vec{x}_\text{CM}$ about the origin and a simultaneous
rotation (in general about a different angle) of $\vec{r}$.
This can be seen graphically in Fig.~\ref{fig_orbit_noether}, where orbits
of the finite transformation $\lambda \mapsto \exp(\lambda B_7) x_0$
are shown.
\begin{figure}
\begin{center}
\centerline{
\includegraphics[width=0.6\textwidth]{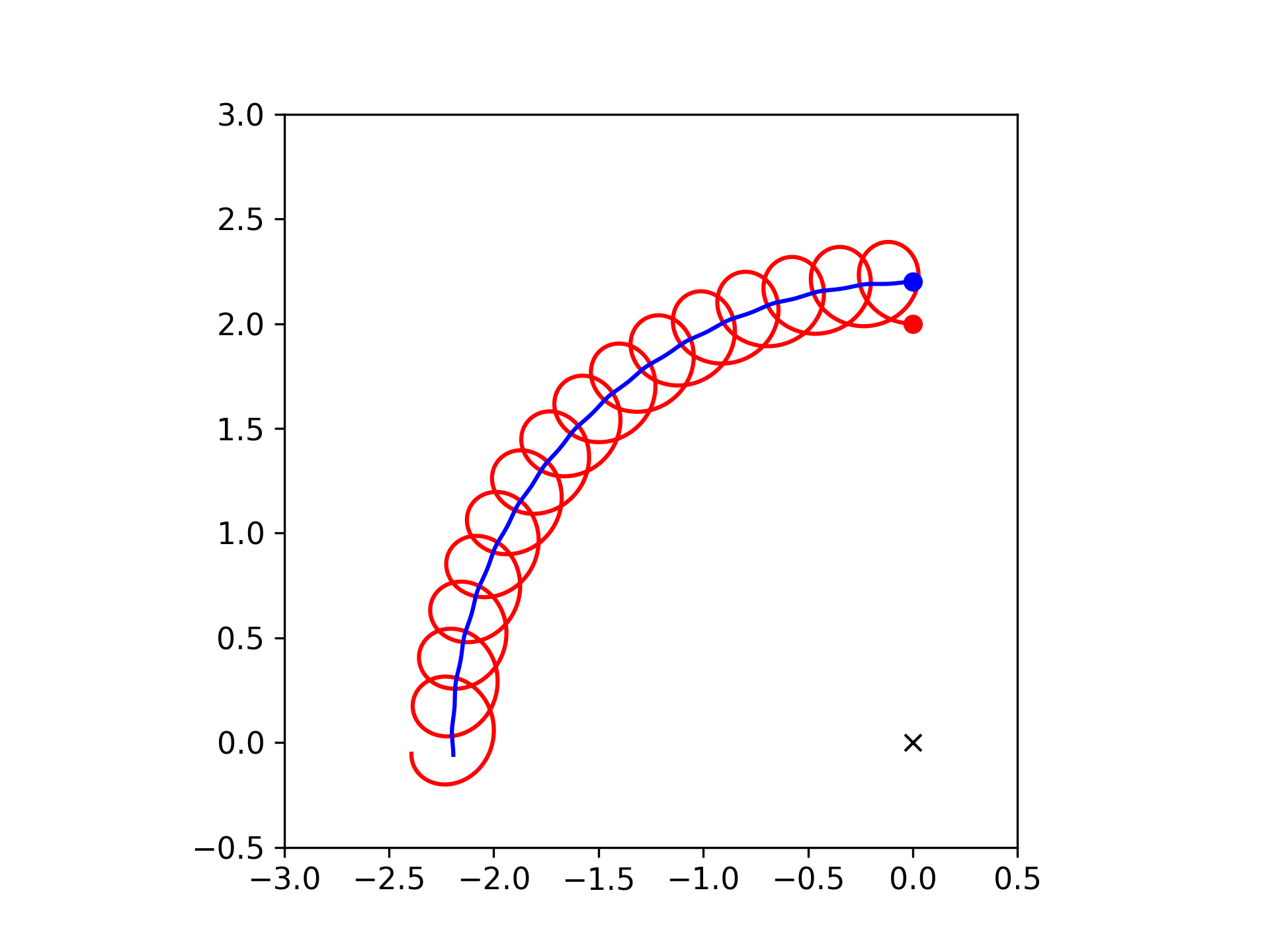}
\hspace*{-2cm}
\includegraphics[width=0.6\textwidth]{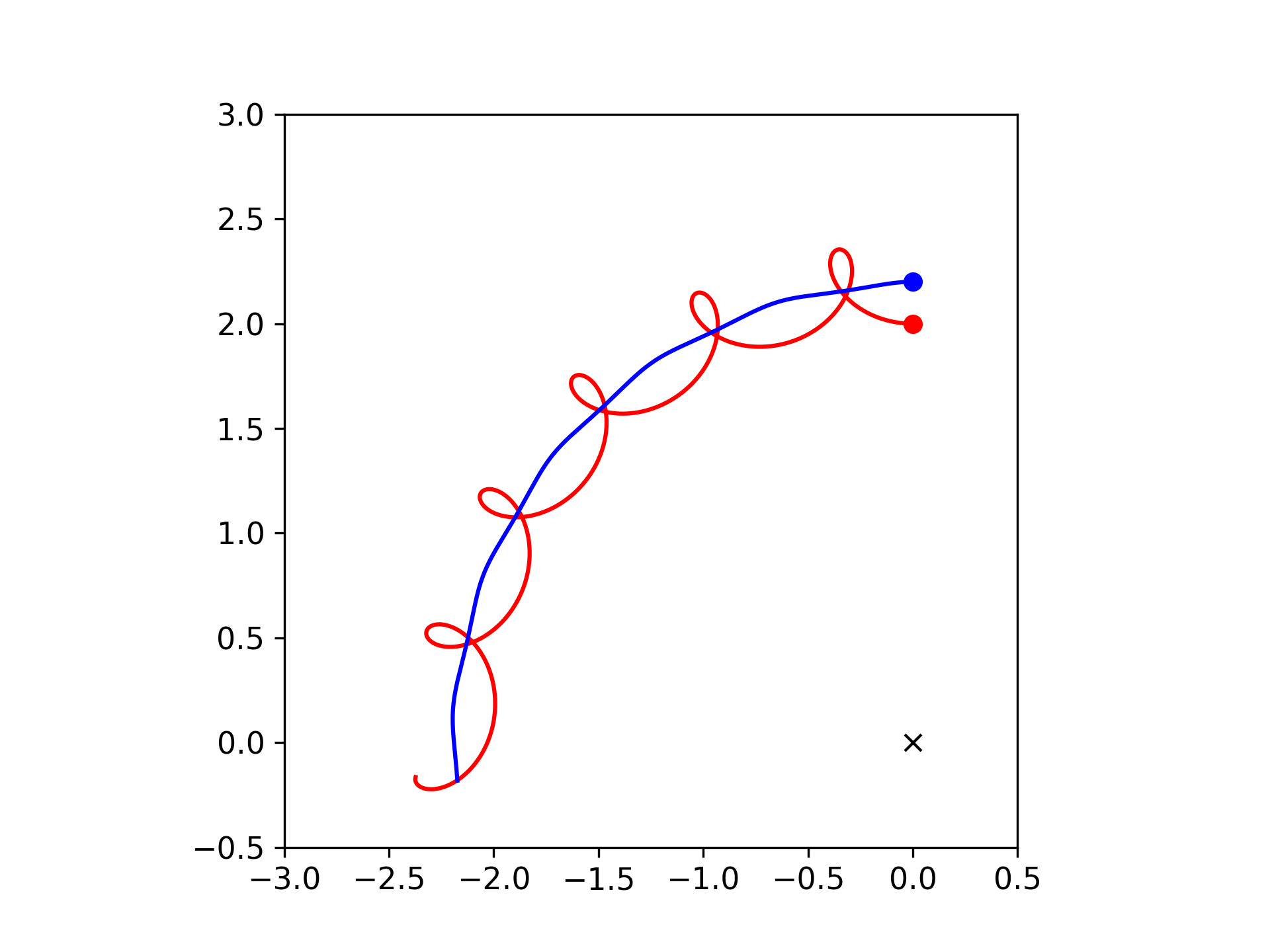}
}
\end{center}
\caption{Orbits $\lambda \mapsto \exp(\lambda B_7) x_0$ for the symmetry transformation
of Eq.~(\ref{eq:a7}) for $m_2/m_1=60$ (left) and $m_2/m_1=20$ (right). The red and blue lines are the orbits of
bodies 1 and 2, respectively. The initial positions $x_0$
of the two bodies are shown as bold points. The length units are arbitrary and
the origin is indicated by a black cross.}\label{fig_orbit_noether}
\end{figure}

Inserting Eqs.~(\ref{eq:a7}), (\ref{eq:a8}) and
(\ref{eq:a9}) into Eq.~(\ref{conserved_quantity}),
the conserved quantities
\begin{equation}
m_1 \dot{\vec{x}}_1 \times \vec{x}_1 + \frac{m_1 m_2}{m_1-m_2}
(\dot{\vec{x}}_1 \times \vec{x}_2 + \dot{\vec{x}}_2 \times \vec{x}_1)
\end{equation}
are found. Their physical interpretation is not evident on first sight.
Eq.~(\ref{eq:transformation_cm_r}), however, suggests that they are
proportional to $(1+\beta) \vec{L}_\mathrm{CM} -\beta \vec{L}_r$,
where $\vec{L}_\mathrm{CM}$ and $\vec{L}_r$
are the components of the angular momentum
associated with the motion of the center of mass
and the relative motion of the two bodies about
the center of mass, respectively.
Indeed, defining
\begin{equation}
\vec{L}_\mathrm{CM} \equiv (m_1+m_2)\,\vec{x}_\mathrm{CM} \times \dot{\vec{x}}_\mathrm{CM},
\quad
\vec{L}_r \equiv \frac{m_1 m_2}{m_1+m_2} \vec{r}\times \dot{\vec{r}}
\end{equation}
one finds
\begin{equation}
m_1 \dot{\vec{x}}_1 \times \vec{x}_1 + \frac{m_1 m_2}{m_1-m_2}
(\dot{\vec{x}}_1 \times \vec{x}_2 + \dot{\vec{x}}_2 \times \vec{x}_1) =
-((1+\beta) \vec{L}_\mathrm{CM} -\beta \vec{L}_r)
\end{equation}
in perfect agreement with Eq.~(\ref{eq:transformation_cm_r}). Since the total angular momentum
$\vec{L} = \vec{L}_\mathrm{CM} + \vec{L}_r$
is conserved too, this implies that $\vec{L}_\mathrm{CM}$ and $\vec{L}_r$ are
conserved \textit{separately}. $\dot{\vec{L}}_r=0$ implies that the
relative motion always happens in the same plane
orthogonal to $\vec{L}_r$.

The key message of this example is that the normal form here
allows to transform a merely numerically generated basis
into a structured form fostering physical interpretation of the numerical results.
Note that starting from standard Cartesian coordinates, the basis in normal
form immediately reveals momentum and angular momentum conservation and,
via the idea to block-diagonalise the matrices $B_7$ to $B_9$, a posteriori \textit{suggests}
the introduction of relative and center-of-mass coordinates.

\subsection{Linear regression}

Consider a linear model
\begin{equation}
y = a^\mathrm{T} x + b
\end{equation}
with model parameters $a\in \mathbbm{R}^p$, $b\in \mathbbm{R}$
and data points $(x, y)\in \mathbbm{R}^p \times \mathbbm{R}$.
W.l.o.g.\ $b$ can be set to zero, as the model may always be reformulated
as
\begin{equation}
y = (a^\mathrm{T}, b)
\begin{pmatrix*}
x\\
1
\end{pmatrix*} \equiv \tilde{a}^\mathrm{T} \tilde{x}.
\end{equation}
Therefore, in the following the simpler model
\begin{equation}
y = a^\mathrm{T} x
\end{equation}
is studied. The standard regression problem is: Given $N$ data points
$(x^{(k)}, y^{(k)})$, $k=1,\ldots,N$, find a best-fit estimate $\hat{a}$
for the model parameters $a$.
This can be done by a least-squares fit:
\begin{equation}
\hat{a} = \underset{a}{\mathrm{argmin}} \sum_{k=1}^N \left( y^{(k)} - a^\mathrm{T} x^{(k)} \right)^2.
\end{equation}
The solution of the above minimisation problem can be found by means of differentiation
w.r.t.\ $a$ and is given by
\begin{equation}\label{eq:hat_a}
\hat{a} = \sum_{k=1}^N \left( \sum_{j=1}^N x^{(j)} x^{(j)\mathrm{T}} \right)^{-1} x^{(k)}  y^{(k)}.
\end{equation}
Collecting the $x^{(k)}$ and $y^{(k)}$ in a $p\times N$-matrix $X$
and an $N\times 1$ vector $Y$,
\begin{equation}
X \equiv (x^{(1)}, \ldots, x^{(N)}), \quad
Y \equiv \begin{pmatrix*}
y^{(1)}\\
\vdots\\
y^{(N)}
\end{pmatrix*},
\end{equation}
Eq.~(\ref{eq:hat_a}) may be rewritten as
\begin{equation}\label{eq:hat_a_pseudoinverse}
\hat{a} = (XX^\mathrm{T})^{-1}X Y.
\end{equation}
This is the well-known result that the least-squares estimate
for the model parameters is given by the Moore-Penrose
pseudo-inverse~\cite{Moore, Bjerhammar, Penrose}
$(XX^\mathrm{T})^{-1}X$ of the input data $X$ times
the output data $Y$.

What is particularly interesting about the solution
in the form of Eq.~(\ref{eq:hat_a_pseudoinverse})
is that it shows that, in many practical cases, a large part
of information contained in the output data $Y$ does
\textit{not contribute at all} to the best-fit value for $a$,
and thus to the generated linear model.
Namely, if the number of data points $N$ exceeds
the number of parameters $p$, in general
\begin{equation}
\mathrm{dim}\, \mathrm{ker}\, (XX^\mathrm{T})^{-1}X = \mathrm{dim}\, \mathrm{ker}\,X = N-p > 0
\end{equation}
will hold. This means that the projection of $Y$
onto $\mathrm{ker}\, X$ will not contribute to
$\hat{a}$.
This is a direct consequence of the linearity of
the model. Since $\hat{a}$ is linear in $Y$,
if $N$ is much larger than $p$, a large
subspace of the space of values for $Y$ must
be mapped to zero by the operator $(XX^\mathrm{T})^{-1}X$.

This issue may be illustrated best by a one-dimensional example.
Consider linear regression for $N$ data points
$(x^{(k)}, y^{(k)}) \in \mathbbm{R}\times \mathbbm{R}$, $k\in\{1,\ldots,N\}$.
The above discussion shows that there are $N-p=N-1$ degrees of freedom
in the output data $Y$ which do not contribute to the best-fit value
$\hat{a}$. These can be parameterised by a basis of $\mathrm{ker}\,X$ which,
assuming $x^{(j)}\neq 0\, \forall j\in \{1,\ldots,N\}$,
in normal form is given by
Eq.~(\ref{nf_orthogonal_vector2}), \textit{i.e.}\
\begin{equation}
(v_i)_j = \begin{cases}
1 & \text{for } j=1\\
-x^{(1)}/x^{(j)} & \text{for } j=i+1,\\
0 & \text{else}
\end{cases}
\end{equation}
for $i\in\{1,\ldots,N-1\}$.
In words this may be stated as: If $y^{(1)}$ is increased by $\Delta y^{(1)}$
and $y^{(k)}$ simultaneously is decreased by
$\Delta y^{(1)}\, x^{(1)}/x^{(k)}$,
the best-fit estimate $\hat{a}$ will remain unchanged.
In symbols:
\begin{equation}\label{eq:Deltay1}
\Delta y^{(1)} x^{(1)} = -\Delta y^{(k)} x^{(k)} \Rightarrow \Delta \hat{a}=0.
\end{equation}
This statement can be generalised to any pair of data points:
Subtracting the basis vectors $v_i$ and $v_k$ in normal form yields
\begin{equation}
(v_i-v_k)_j = \begin{cases}
-x^{(1)}/x^{(j)} & \text{for } j=i+1,\\
+x^{(1)}/x^{(j)} & \text{for } j=k+1,\\
0 & \text{else}
\end{cases}
\end{equation}
from which, after a renaming of indices, the following analogue of Eq.~(\ref{eq:Deltay1})
arises:
\begin{equation}\label{eq:Deltayj}
\Delta y^{(j)} x^{(j)} = -\Delta y^{(k)} x^{(k)} \Rightarrow \Delta \hat{a}=0.
\end{equation}
This is remindful of the \textit{law of the lever}
of classical mechanics. And this is not a mere coincidence.
Consider a lever in the $(x, y)$-plane, supported at $(x, y)=0$.
Its position shall be on the line $y=ax$.
Moreover, assume that forces
\begin{equation}
\vec{F}_i = \begin{pmatrix}
0\\
y^{(i)} - a x^{(i)}\\
0
\end{pmatrix}
\end{equation}
act on it at positions
\begin{equation}
\vec{x}_i = \begin{pmatrix}
x^{(i)}\\
a x^{(i)}\\
0
\end{pmatrix}.
\end{equation}
That is, the farther
away the lever from the point $(x^{(i)}, y^{(i)}, 0)$, the stronger the force on
$\vec{x}_i$ will be. This is illustrated in Fig.~\ref{fig_lever}.
\begin{figure}
\begin{center}
\includegraphics[width=0.6\textwidth]{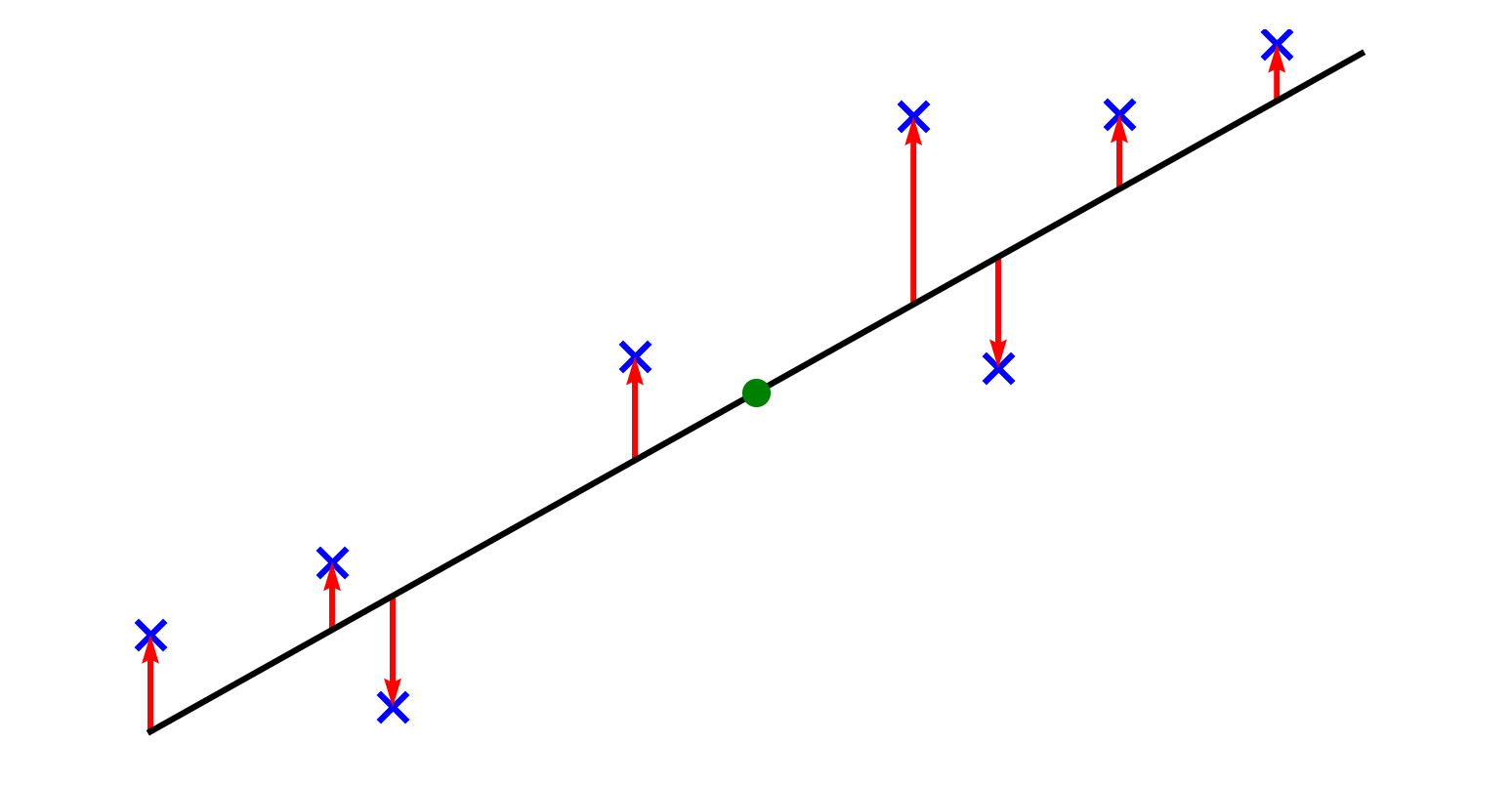}
\end{center}
\caption{The mechanical analogy to a one-dimensional least-squares fit. The forces on the lever (red vectors)
point from the lever positions $(x^{(i)}, a x^{(i)})$
to the data points $(x^{(i)}, y^{(i)})$. The support of the
lever (the green point) is at the origin.}\label{fig_lever}
\end{figure}
The total torque on the lever is given by
\begin{equation}
\vec{\tau} = \sum_{i=1}^N \vec{x}_i \times \vec{F}_i = 
\begin{pmatrix}
0\\
0\\
\sum_{i=1}^N x^{(i)}(y^{(i)}-ax^{(i)})
\end{pmatrix}.
\end{equation}
In mechanical equilibrium the lever will assume a position of zero
torque, \textit{i.e.}
\begin{equation}
\vec{\tau} = 0 \Leftrightarrow \sum_{i=1}^N x^{(i)}(y^{(i)}-ax^{(i)})
= XY-aXX^\mathrm{T} = 0.
\end{equation}
The position, described by the lever's slope $a$, in equilibrium is thus given by
\begin{equation}
a_\mathrm{eq} = (XX^\mathrm{T})^{-1} XY,
\end{equation}
which coincides with best-fit solution $\hat{a}$ for the regression problem.
The one-dimensional regression problem therefore behaves like a mechanical
lever.

The power of the normal form in this example is the
as-strong-as-possible decoupling of the involved quantities (the positions $x^{(j)}$)
in the basis vectors of the space of outputs leaving $\hat{a}$ unchanged.
Breaking the changes of $y$ down to changes of two outputs only allowed
to formulate Eqs.~(\ref{eq:Deltay1}) and~(\ref{eq:Deltayj}) and led
to the discovery of a simple mechanical analogy sharing some aspects of the
mathematical problem.

The normal form can also give interesting insights into regression
problems in higher dimensions.
A mathematical problem related to linear regression is the estimation
of gradients of functions by finite differences.
Suppose the values $f(x)$ of a function at positions
$x^{(i)}\in\mathbbm{R}^2$ are given and the goal is to estimate the gradient
of $f$ at $x^{(0)}$. This can be done by means of a linear model
\begin{equation}
f(x) \approx f(x^{(0)}) + \nabla f\vert_{x^{(0)}}^\mathrm{T} (x-x^{(0)}).
\end{equation}
The best-fit estimate based on this model is given by Eq.~(\ref{eq:hat_a_pseudoinverse}),
\textit{i.e.}\
\begin{equation}
\widehat{\nabla f}\vert_{x^{(0)}} = (XX^\mathrm{T})^{-1}X Y,
\end{equation}
where
\begin{equation}
X=(x^{(1)}-x^{(0)},\ldots, x^{(N)}-x^{(0)}) \quad \text{and}\quad
Y = \begin{pmatrix*}
f(x^{(1)})-f(x^{(0)})\\
\vdots\\
f(x^{(N)})-f(x^{(0)})
\end{pmatrix*}.
\end{equation}
The directions in the space of values of $f$ which do not contribute
to the estimate of $\nabla f$ can thus be parameterised by a basis of $\mathrm{ker}\,X$.
It is interesting to apply this to finite difference schemes for estimating the gradient.
For the forward differences
\begin{equation}
\frac{\partial f}{\partial x_i}\vert_{x^{(0)}} \approx \frac{1}{\delta} \left( f(x^{(0)} + \delta e_i) - f(x^{(0)}) \right), \quad i\in\{1, 2\}
\end{equation}
with $e_i$ denoting the unit vector in $i$-direction one has
$x^{(1)}=x^{(0)}+\delta e_1$, $x^{(2)}=x^{(0)}+\delta e_2$.
Thus,
\begin{equation}
X = \begin{pmatrix*}
\delta & 0\\
0 & \delta
\end{pmatrix*}
\end{equation}
and $\mathrm{ker}\,X=\{0\}$. This reflects the fact that all three function evaluations
at $x^{(0)}$, $x^{(1)}$ and $x^{(2)}$
are needed to estimate the gradient using the forward difference. The same holds true
for the backward difference too.
A more interesting case is given by the finite difference scheme
illustrated in Fig.~\ref{fig_finite_differences}.
\begin{figure}
\begin{center}
\includegraphics[width=0.5\textwidth]{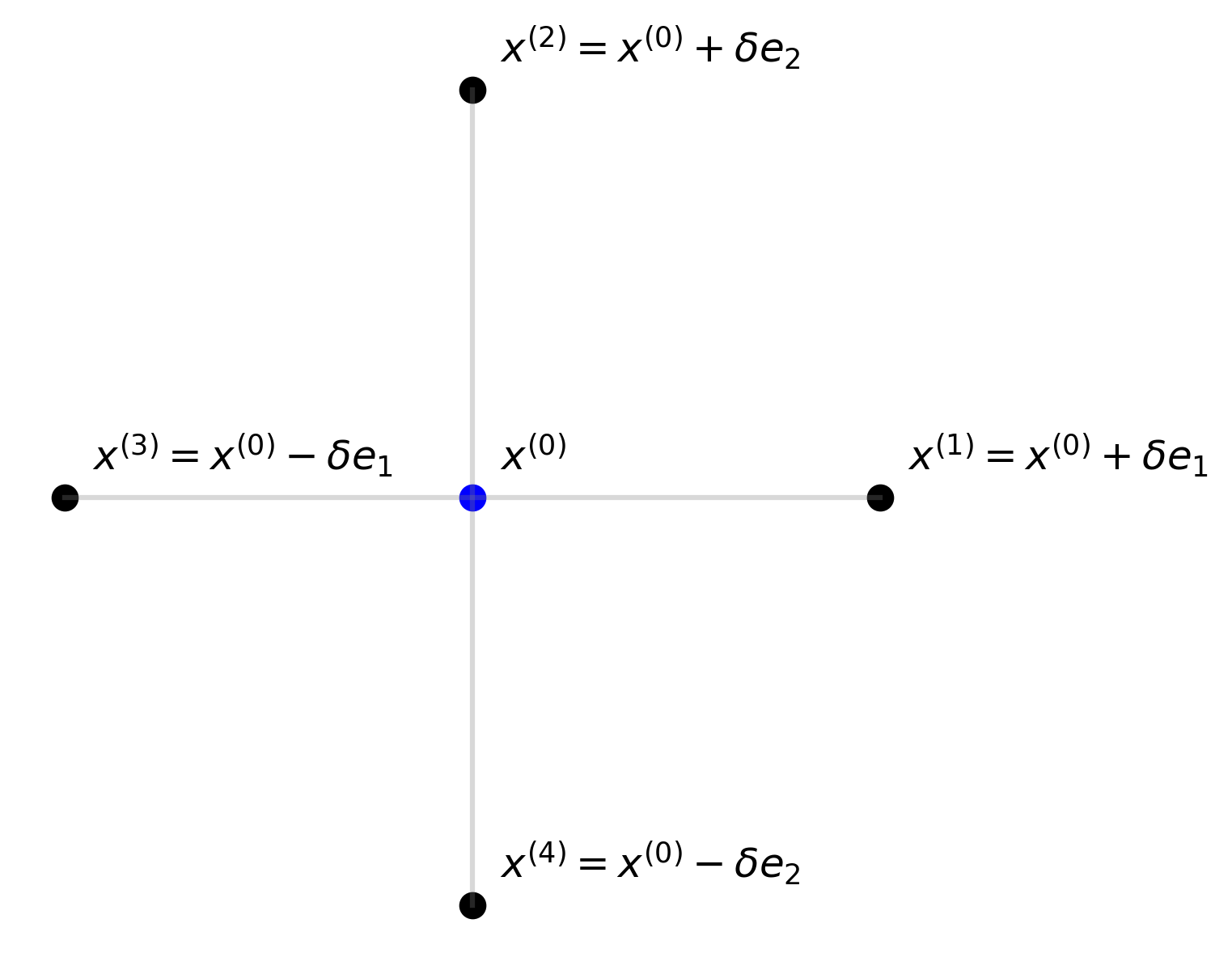}
\end{center}
\caption{A finite difference scheme based on five function evaluations.}\label{fig_finite_differences}
\end{figure}
For this case
\begin{equation}
X = \delta \begin{pmatrix*}[r]
1 & 0 & -1 & 0\\
0 & 1 & 0 & -1
\end{pmatrix*}
\end{equation}
In normal form the basis of $\mathrm{ker}\, X$ is given by
\begin{equation}
(v_1, v_2) = \begin{pmatrix*}
1 & 0\\
0 & 1\\
1 & 0\\
0 & 1
\end{pmatrix*}.
\end{equation}
From this the two independent combinations of function evaluations \textit{not} contributing
to the gradient estimate are found to be
\begin{subequations}\label{eq_finite_diff}
\begin{align}
\begin{split}
 v_1^\mathrm{T} Y & = f(x^{(1)}) - f(x^{(0)}) + f(x^{(3)}) - f(x^{(0)})\\
                  & = f(x^{(0)}+\delta e_1) - 2f(x^{(0)}) + f(x^{(0)} - \delta e_1),
\end{split}\\
\begin{split}
 v_2^\mathrm{T} Y & = f(x^{(2)}) - f(x^{(0)}) + f(x^{(4)}) - f(x^{(0)})\\
                  & = f(x^{(0)}+\delta e_2) - 2f(x^{(0)}) + f(x^{(0)} - \delta e_2).
\end{split}
\end{align}
\end{subequations}
Since these linear combinations provide information going beyond the mere gradient
information, they must contain information about higher order derivatives of $f$.
Indeed, by Taylor expansion of $f$ one finds
\begin{subequations}\label{eq_finite_diff2}
\begin{align}
& \frac{1}{\delta^2} v_1^\mathrm{T} Y \xrightarrow[\delta \to 0]{} \frac{\partial^2 f}{\partial x_1^2}\vert_{x^{(0)}},\\
& \frac{1}{\delta^2} v_2^\mathrm{T} Y \xrightarrow[\delta \to 0]{} \frac{\partial^2 f}{\partial x_2^2}\vert_{x^{(0)}}.
\end{align}
\end{subequations}
That is, the two basis vectors in normal form correspond to the two second order derivatives
which can be estimated given the function evaluations of the scheme of Fig.~\ref{fig_finite_differences}.

\section{Conclusions}\label{sec_conclusions}

The purpose of the normal form for bases proposed in this paper
is to provide a technique to support the understanding of the
structure of the problem in which the basis occurs.
Sec.~\ref{sec_applications} illustrates the power
of this technique in sample applications from different areas of
physics and mathematics.

From the mathematical side, the most important properties required for the normal
form are uniqueness and a high number of zero entries in the basis vectors.
A widely used normal form for matrices which shares uniqueness and which also
maps bases to bases is the reduced row echelon form (rref)
which can be computed by Gau\ss-Jordan elimination---see textbooks
on linear algebra.
Given a basis of a finite-dimensional vector space as \textit{rows} of a matrix,
the rref fulfils
\begin{itemize}
 \item All leading entries (leftmost non-zero entries) are one.
 \item Every column with a leading entry has zeros in all other places.
 \item The zero rows, if there are any, are below all non-zero rows.
\end{itemize}
This looks similar to the requirements on the normal form proposed in this paper.
Indeed, for $A$ being a matrix of full column rank
\begin{equation}
A \mapsto \mathrm{rref}(A^\mathrm{T})^\mathrm{T},
\end{equation}
which is the \textit{reduced column echelon form} (rcef), provides a unique normal form
for the basis of the vector space spanned by the columns of $A$.
It, however, does not share the goal of having as many zero entries as possible and
in general will not lead to the highest possible number of zero entries.
As an illustration, take the basis of Eq.~(\ref{ker_B_numpy}).
Its rcef and its normal form are
\begin{equation}\label{rref_ker_dim_analysis_B}
\text{rcef: } \left(
 \begin{smallmatrix*}[r]
 1 &  0\\
 0 &  1\\
 0 & -1\\
 1/2 & 1/2\\
-1/2 & -1/2
 \end{smallmatrix*}
 \right), \quad
\text{normal form: }
 \left(
 \begin{smallmatrix*}[r]
 1 & 1\\
-1 & 0\\
 1 & 0\\
 0 & 1/2\\
 0 & -1/2
 \end{smallmatrix*}
 \right).
\end{equation}
Though both forms are simple and structured, the rcef has less zero entries than the normal form.
A case in which the rcef and the normal form strongly differ may be constructed as follows.
Suppose $A$ and $B$ are two different, e.g.\ random, invertible $n\times n$-matrices and let
the columns of the matrix
\begin{equation}
\begin{pmatrix*}
A\\
B\\
B\\
\vdots\\
B
\end{pmatrix*}
\end{equation}
denote the basis of an $n$-dimensional vector space. Then the rcef will be given by
\begin{equation}
\begin{pmatrix*}
\mathbbm{1}\\
BA^{-1}\\
BA^{-1}\\
\vdots\\
BA^{-1}
\end{pmatrix*}
\end{equation}
which will have many zero entries less than the normal form, which
must have at least as many zeros as
\begin{equation}
\begin{pmatrix*}
AB^{-1}\\
\mathbbm{1}\\
\mathbbm{1}\\
\vdots\\
\mathbbm{1}
\end{pmatrix*}.
\end{equation}
However, there is a tradeoff. The rcef is much easier to compute than
the normal form. Therefore, in cases where the normal form is too expensive to compute,
the rcef may be a good alternative.

To summarise, the normal form presented in this paper can be a powerful tool
for unveiling the structure of problems involving bases of finite-dimensional
vector spaces. An algorithm for its computation as well as an implementation
for this algorithm are given in appendices~\ref{sec:algorithm}
and~\ref{sec_python_implementation}, respectively, and the readers
are encouraged to use these to gain insight into the structure of their own problems.

\appendix

\section{An algorithm for computing the normal form}\label{sec:algorithm}

The simplest algorithm for constructing the normal form is to 
generate all
\begin{equation}
\left(\begin{matrix}
m\\
n-1
\end{matrix}\right) = \frac{m!}{(n-1)!\,(m-n+1)!}
\end{equation}
selections of $n-1$ rows of $A$, check whether the selected rows
span an $(n-1)$-dimensional space and if so, construct the
corresponding $\hat{s}$. Ordering the $\hat{s}$ then
allows to construct the normal form.
In the following this algorithm will be referred to as the
\textit{standard} algorithm. A possibility for
such an algorithm is shown in the left part of
Fig.~\ref{fig_algorithm}.
In practice, the normal form will be most interesting
in cases where there are many more zero entries achievable
in the basis vectors than the theoretical minimum of $n-1$.
If the number of achievable zeros is close to the theoretical
maximum $m-1$, an algorithm trying to construct $\hat{s}$ leading
to $m-1, m-2,\ldots$ zeros in the basis vectors may
be sufficiently faster in constructing the normal form.
Namely, using this strategy, if after finishing
the trial of all possibilities for $m-k$ zeros $n$
linearly independent $\hat{s}$ are found, the search
can be stopped, since all not yet found $\hat{s}$
will have smaller $\theta(\hat{s})$.
The number of selections to be tried
is then
\begin{equation}
\sum_{j=1}^k\left(\begin{matrix}
m\\
m-j
\end{matrix}\right).
\end{equation}
For small
$k$ this search strategy will be advantageous. Since
this algorithm searches ``from the top'' (highest possible
number of zero entries first), in the following it will be referred to
as the \textit{top-down} algorithm. It is illustrated in the
right part of Fig.~\ref{fig_algorithm}.
\begin{figure}
\begin{center}
\includegraphics[height=0.9\textheight]{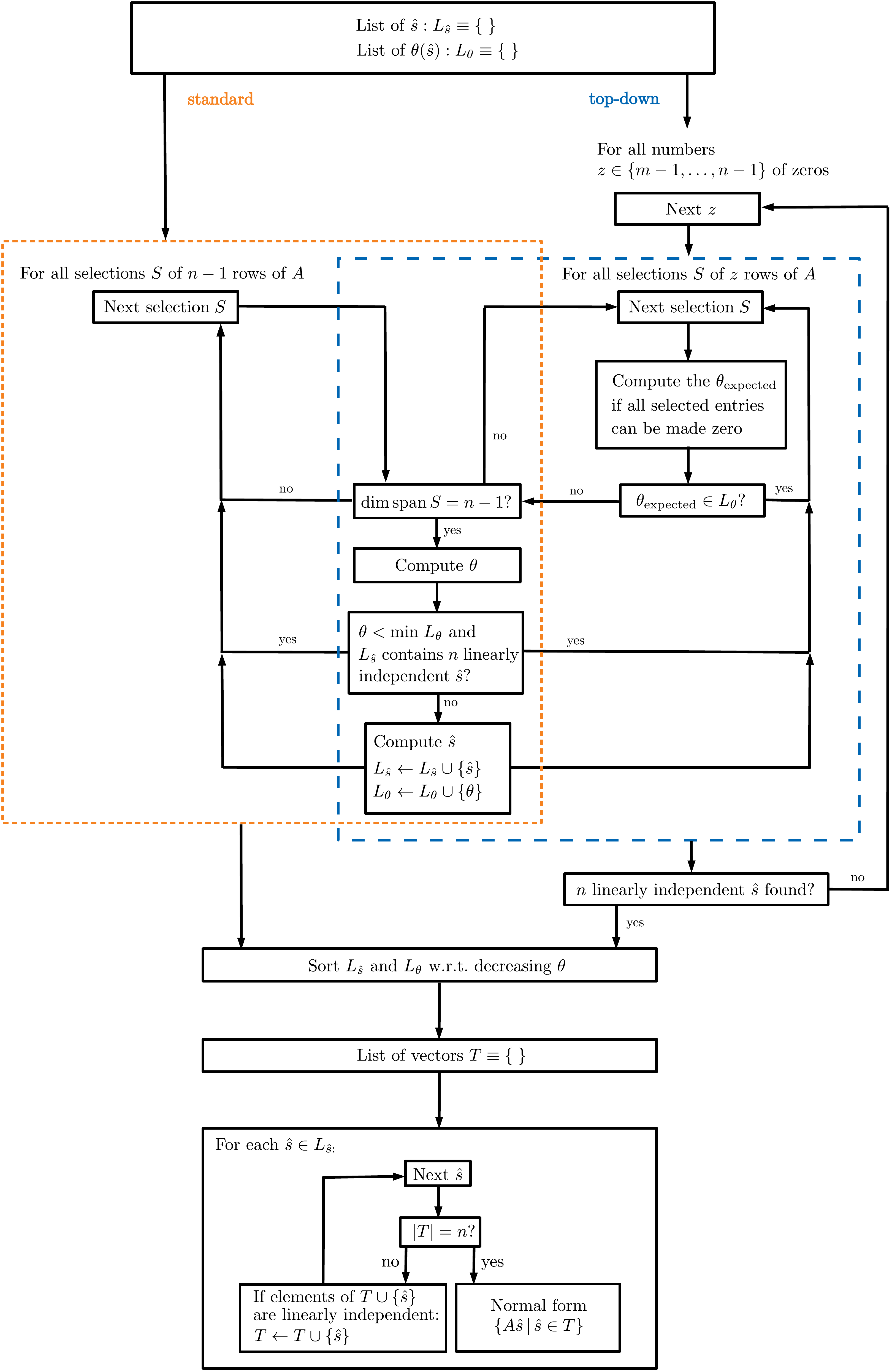}
\end{center}
\caption{An algorithm for computing the normal form. Left (fine dashed): \textit{standard} algorithm, right (dashed): \textit{top-down}
algorithm.}\label{fig_algorithm}
\end{figure}
In appendix~\ref{sec_python_implementation} a Python~\cite{python}
implementation of both algorithms is provided.

\section{A Python implementation of the normal form}\label{sec_python_implementation}

The following Python implementation allows to compute the normal form
of real and complex matrices of full column rank.
The function computing the normal form is defined in line 119 of the
source code below, and an example for its use can be found following line 133.

Computation times for random matrices of sizes
$20\times 10$ to $20 \times 19$
with the standard algorithm range from less than
a second to about 45 minutes on a typical personal computer.
This is sufficient for the examples presented in
this paper.
For matrices which have many zero entries, the top-down
algorithm will in general be much faster.
As an illustration, consider the normal form of the basis of the example
of Sec.~\ref{sec:Noether} (see also Fig.~\ref{fig_basis_nf_noether}).
The basis is described by a $42\times 9$-matrix with twelve
vanishing rows, effectively giving rise to a $30\times 9$
matrix. A very crude estimate for the running time of the
standard algorithm in this case is
\begin{equation}
t_\text{standard} \sim \left(\begin{matrix}
30\\
8
\end{matrix}\right) t_\mathrm{SVD},
\end{equation}
where $t_\mathrm{SVD}$ is the typical time needed
for one singular value decomposition in line
6 of the source code below.
The resulting normal form has at most six non-zero entries
in a column, see Eq.~(\ref{kernel_basis_noether}). Thus, for
the top-down algorithm the according estimate is
\begin{equation}
t_\text{top-down} \sim \sum_{j=1}^6\left(\begin{matrix}
30\\
30-j
\end{matrix}\right) t_\mathrm{SVD},
\end{equation}
giving the estimate
\begin{equation}
\frac{t_\text{standard}}{t_\text{top-down}} \sim 7.6.
\end{equation}
The measured times on a typical personal computer are
$t_\text{standard}=478.1\,\text{s}$ and $t_\text{top-down}=111.3\,\text{s}$,
\textit{i.e.}\ $t_\text{standard} / t_\text{top-down} \approx 4.3$.

A rigorous analysis of the asymptotic complexities
of the two presented algorithms, as well as a discussion of
possible improvements is beyond the scope of this paper.
The key message here shall be that
for moderately sized matrices (up to $m, n\lesssim 15$) and even for
larger examples like the $30\times 9$-matrix from Sec.~\ref{sec:Noether},
the presented algorithms and the provided implementation
allow computation of the normal form in seconds to minutes.

\newpage
\newgeometry{textwidth=16cm}
%
%
\begin{scriptsize}
\begin{verbatim}
The code below is licensed under the MIT license:

Copyright (c) 2023 Patrick Otto Ludl

Permission is hereby granted, free of charge, to any person obtaining a copy
of this software and associated documentation files (the "Software"), to deal
in the Software without restriction, including without limitation the rights
to use, copy, modify, merge, publish, distribute, sublicense, and/or sell
copies of the Software, and to permit persons to whom the Software is
furnished to do so, subject to the following conditions:

The above copyright notice and this permission notice shall be included in all
copies or substantial portions of the Software.

THE SOFTWARE IS PROVIDED "AS IS", WITHOUT WARRANTY OF ANY KIND, EXPRESS OR
IMPLIED, INCLUDING BUT NOT LIMITED TO THE WARRANTIES OF MERCHANTABILITY,
FITNESS FOR A PARTICULAR PURPOSE AND NONINFRINGEMENT. IN NO EVENT SHALL THE
AUTHORS OR COPYRIGHT HOLDERS BE LIABLE FOR ANY CLAIM, DAMAGES OR OTHER
LIABILITY, WHETHER IN AN ACTION OF CONTRACT, TORT OR OTHERWISE, ARISING FROM,
OUT OF OR IN CONNECTION WITH THE SOFTWARE OR THE USE OR OTHER DEALINGS IN THE
SOFTWARE.
\end{verbatim}
\end{scriptsize}
\begin{lstlisting}[language=Python]
import numpy as np
import itertools


def compute_rank_and_kernel(A: np.ndarray, accuracy):
    _, s, Vh = np.linalg.svd(A, full_matrices=True)
    rank = np.sum(np.abs(s) > accuracy)
    kernel = Vh[rank:, :].transpose()
    return rank, kernel


def are_linearly_independent(vectors: list, accuracy):
    """vectors: List of 1-dim. numpy arrays"""
    rank = np.linalg.matrix_rank(np.array(vectors).transpose(), accuracy)
    return rank == len(vectors)


def get_theta2_from_selection(selection: list):
    """selection [list(0 or 1)]: list which indicates which rows of the matrix
       are selected"""
    theta2 = 0
    for k in range(0, len(selection)):
        theta2 += selection[k] * 2**k
    return theta2


class _NormalForm:
    def __init__(self, A: np.ndarray, accuracy, top_down=False):
        self.A = A
        self.accuracy = accuracy
        self.top_down = top_down
        self.m = A.shape[0]
        self.n = A.shape[1]
        self.list_of_s_hat = []
        self.list_of_theta = []
        self.theta_min = 0
        self.n_linearly_independent_s_hat_found = False

    def _compute_theta(self, element_is_zero):
        theta2 = get_theta2_from_selection(list(element_is_zero))
        return 2**(self.m+sum(element_is_zero)) + theta2

    def _update(self, theta, s, As, element_is_zero):
        if theta not in self.list_of_theta:
            index = np.where(np.logical_not(element_is_zero))[0][0]
            s_hat = s / As[index]
            self.list_of_theta = self.list_of_theta + [theta]
            self.theta_min = min(self.list_of_theta)
            self.list_of_s_hat = self.list_of_s_hat + [s_hat]
            if not self.n_linearly_independent_s_hat_found:
                if len(self.list_of_s_hat) >= self.n:
                    rank = np.linalg.matrix_rank(
                        np.concatenate(self.list_of_s_hat, axis=1),
                        self.accuracy)
                    self.n_linearly_independent_s_hat_found = (rank == self.n)

    def _compute_transformation_matrix(self):
        sorting_indices = list(reversed(sorted(range(len(self.list_of_theta)),
                                        key=lambda k: self.list_of_theta[k])))
        self.list_of_s_hat = [self.list_of_s_hat[i] for i in sorting_indices]
        transformation_matrix = [self.list_of_s_hat[0]]
        for k in range(1, len(self.list_of_s_hat)):
            if len(transformation_matrix) == self.n:
                break
            if are_linearly_independent(transformation_matrix +
                                        [self.list_of_s_hat[k]],
                                        self.accuracy):
                transformation_matrix = (transformation_matrix +
                                         [self.list_of_s_hat[k]])
        return np.concatenate(transformation_matrix, axis=1)

    def _normalform(self):
        for selection in itertools.combinations(range(0, self.m), self.n-1):
            rank, s = compute_rank_and_kernel(self.A[selection, :],
                                              self.accuracy)
            if rank != self.n-1:
                continue
            As = (self.A@s).flatten()
            element_is_zero = (np.abs(As) < self.accuracy)
            theta = self._compute_theta(element_is_zero)
            if (self.n_linearly_independent_s_hat_found
                    and theta < self.theta_min):
                continue
            self._update(theta, s, As, element_is_zero)

    def _normalform_top_down(self):
        for number_of_zeros in reversed(range(self.n-1, self.m)):
            for selection in itertools.combinations(range(0, self.m),
                                                    number_of_zeros):
                a = np.zeros(self.m).astype(np.int32)
                a[list(selection)] = 1
                theta2_expected = get_theta2_from_selection(list(a))
                theta_expected = 2**(self.m+sum(a)) + theta2_expected
                if theta_expected in self.list_of_theta:
                    continue
                rank, s = compute_rank_and_kernel(self.A[selection, :],
                                                  self.accuracy)
                if rank != self.n-1:
                    continue
                As = (self.A@s).flatten()
                element_is_zero = (np.abs(As) < self.accuracy)
                theta = self._compute_theta(element_is_zero)
                if (self.n_linearly_independent_s_hat_found and
                        theta < self.theta_min):
                    continue
                self._update(theta, s, As, element_is_zero)
            if self.n_linearly_independent_s_hat_found:
                break

    def compute(self):
        if self.top_down:
            self._normalform_top_down()
        else:
            self._normalform()
        transformation_matrix = self._compute_transformation_matrix()
        return self.A@transformation_matrix, transformation_matrix


def normalform(A: np.ndarray, accuracy=1.0e-8, top_down=False):
    rankA = np.linalg.matrix_rank(A, accuracy)
    if rankA < A.shape[1]:
        raise ValueError("Error: The normal form is only defined " +
                         "for matrices of full column rank. The matrix is " +
                         str(A.shape[0]) + " x " + str(A.shape[1]) +
                         " and its rank is " + str(rankA)+".")
    # Check for zero rows in A. These can be neglected!
    row_norms = np.linalg.norm(A, axis=1)
    A_reduced = A[row_norms > accuracy, :]
    _, transformation_matrix = _NormalForm(A_reduced, accuracy, top_down).compute()
    return A@transformation_matrix, transformation_matrix


if __name__ == "__main__":
    A = np.array([[1, 1], [2, 2], [3, 3], [1, 2]])
    print(A)

    A_nf, _ = normalform(A)
    A_nf_expected = np.array([[0, 1], [0, 2], [0, 3], [1, 0]])
    assert np.linalg.norm(A_nf-A_nf_expected) < 1.0e-12

    def chop(A: np.ndarray, limit=1.0e-15):
        A_copy = A.copy()
        A_copy[np.abs(A_copy) < limit] = 0
        return A_copy

    print(chop(A_nf))
\end{lstlisting}
\restoregeometry

\newpage
\newgeometry{a4paper,textwidth=140mm,lmargin=20mm}

\end{document}